\documentclass{amsart}
\usepackage{amsfonts,pdfsync,color,graphicx,calligra,mathrsfs,tikz}
\usepackage{amsthm,amssymb}

\usepackage[english]{babel}			
\usepackage[latin1]{inputenc}
\usepackage[T1]{fontenc}

\newtheorem*{theorem*}{Theorem}

\title{A short proof of smooth implies flat}
\author{Jes\'us Conde--Lago}
\date{}
\address{Departamento de \'Alxebra, Facultade de Matem\'aticas, Universidade de Santiago de Compostela, E-15782 Santiago de Compostela, Spain}
\email{jesus.conde@usc.es}

\begin{document}
	\vspace{-5ex}
	\begin{abstract}
		Proofs that a smooth morphism is flat available in the literature are long and difficult. We give a short proof of this fact.
	\end{abstract}
	\vspace{5ex}
	
	\maketitle
	
	Let $(A,\mathfrak{m})\rightarrow (B,\mathfrak{n})$ be a local homomorphism of noetherian local rings. It is said \cite[$0_{IV}$.19.3.1]{EGA} that $B$ is formally smooth over $A$ (for the $\mathfrak{n}$-adic topology) if for any $A$-algebra $C$ and any nilpotent ideal $N$ of $C$, each $A$-algebra homomorphism $u: B\rightarrow C/N$ satisfying $u(\mathfrak{n}^k)=0$ for some $k$, factorizes as $B\overset{v}{\rightarrow} C \overset{p}{\rightarrow} C/N$ where $p$ is the canonical map.
	
	A fundamental result, proved by Grothendieck in \cite[$0_{IV}$.19.7.1]{EGA}, says that if $B$ is formally smooth over $A$ then $B$ is flat over $A$. His proof is long. In the following few years shorter proofs appeared, but based on results whose proofs were long.
	
	In many situations, one is concerned exclusively with \emph{smooth} algebras, i.e., formally smooth algebras essentially of finite type \cite[17.3.1, 17.1.2]{EGA} (which are then formally smooth for the discrete topology). In this case, a shorter proof of this result can be seen in \cite[\S7 n.10, Lemme 5]{Bo}, but it uses non-trivial results on smoothness and regularity. There are also two short papers deducing flatness from the Jacobian criterion of smoothness in some particular cases (for an extension of polynomial algebras over a field in \cite{MRW} and in characteristic zero in \cite{MR}).
	
	In this paper we deduce flatness from the very definition of smoothness in a short way.
	
	\begin{theorem*}
		Let $A$ be a noetherian ring and $B$ a smooth $A$-algebra. Then $B$ is flat as $A$-module.
	\end{theorem*}
	\begin{proof}
		Let $R$ be a localization of a polynomial $A$-algebra of finite type such that we have a surjective homomorphism $\varphi: R\rightarrow B$. Let $I=\operatorname{ker}\varphi$. Then $\hat{R}/\hat{I}=R/I=B$ and so $\hat{I}=\operatorname{ker}(\hat{\varphi}:\hat{R}\to B)$. As in \cite[$0_{\text{IV}}$.19.3.11]{EGA} we are going to see that we have a section $B\rightarrow \hat{R}$. By induction on $i\geq 2$, we construct $A$-algebra homomorphisms $f_i: B\rightarrow \hat{R}/\hat{I}^i$
		\begin{center}
			\begin{tikzpicture}[scale=2]
			\node (A) at (0.1,1) {$0$};
			\node (B) at (1,1) {$\hat{I}^{i-1}/\hat{I}^i$};
			\node (C) at (2,1) {$\hat{R}/\hat{I}^i$};
			\node (D) at (3,1) {$\hat{R}/\hat{I}^{i-1}$};
			\node (E) at (3.9,1) {$0$};
			\node (F) at (3,1.8) {$B$};
			\path[->,font=\scriptsize]
			(A) edge node[above]{} (B)
			(B) edge node[left]{} (C)
			(C) edge node[above]{} (D)
			(D) edge node[above]{} (E)
			(F) edge[densely dashed] node[left]{$f_i$} (C)
			(F) edge node[right]{$f_{i-1}$} (D);
			\end{tikzpicture}
		\end{center}
		making commutative the triangles. Since $\hat{R}$ is complete for the $\hat{I}$-adic topology, we have a limit homomorphism $$f=\varprojlim_{i} f_i: B \longrightarrow \varprojlim_{i} \hat{R}/\hat{I}^i=\hat{R}$$ such that $\hat{\varphi}f=\operatorname{id}_{B}$. Therefore $B$ is a direct summand of the flat $A$-module $\hat{R}$ and thus also flat.
	\end{proof}

\end{document}